\documentclass[12pt]{amsart}
\usepackage{amssymb, amscd, amsmath, amsthm, epsfig, latexsym, enumerate}

\renewcommand{\geq}{\geqslant}
\renewcommand{\leq}{\leqslant}

\newtheorem{theorem}{Theorem}
\newtheorem{lemma}[theorem]{Lemma}
\newtheorem{cor}[theorem]{Corollary}

\newtheorem*{cor*}{Corollary}

\begin{document}

\title{domination by geometric 4-manifolds}

\author{Jonathan A. Hillman }
\address{School of Mathematics and Statistics\\
     University of Sydney, NSW 2006\\
      Australia }

\email{jonathanhillman47@gmail.com}

\begin{abstract}
We consider aspects of the question
``when is an orientable closed 4-manifold $Y$ dominated by another 
such manifold $X$?",
focusing on the cases when $X$ is geometric or fibres non-trivially over
an orientable surface.
\end{abstract}


\maketitle
S.-C. Wang has considered in detail properties of maps of non-zero degree 
between 3-manifolds.
In particular, he grouped aspherical 3-manifolds into 8 families, 
according to the nature of their JSJ decompositions,
and determined which pairs allowed maps of non-zero degree between representative 3-manifolds \cite{Wa91}.
Purely algebraic arguments for $PD_n$-groups with JSJ decompositions 
and all $n$ were given in \cite{Hi04}.
Sharper results for maps between
aspherical geometric 4-manifolds were given in \cite{Ne}.
We shall complement this work in dimension 4 by considering cases 
where the domain is geometric but not aspherical, 
or fibres non-trivially over a surface. 
In the latter case the strongest results are when the range is aspherical.

We begin by reviewing the results of \cite{Hi04} and \cite{Ne} for aspherical geometric 4-manifolds.
(Some of these earlier results are recovered below in passing.)
In \S2 we make some basic observations and give four simple lemmas.
The remaining sections are organized in terms of the geometry of the
dominating space $X$.
The geometries $\mathbb{S}^4$,  $\mathbb{CP}^2$, 
$\mathbb{S}^2\times\mathbb{S}^2$,  $\mathbb{S}^3\times\mathbb{E}^1$, 
$\mathbb{S}^2\times\mathbb{E}^2$
and $\mathbb{S}^2\times\mathbb{H}^2$ are considered
in \S3, and the geometries of solvable Lie type in \S4.
All total spaces of bundles with base and fibre $S^2$ or the torus $T$
have such geometries, excepting only $S^2\tilde\times{S^2}$.
In \S5 we consider domination by total spaces of $T$-bundles with
base a hyperbolic surface. 
Among these are manifolds with geometry $\mathbb{H}^2\times\mathbb{E}^2$
or $\widetilde{\mathbb{SL}}\times\mathbb{E}^1$,
but there are also non-geometric $T$-bundle spaces.
The next section considers aspherical domains which are (virtually)
both $S^1$-bundle spaces and mapping tori.
(This includes $\mathbb{H}^3\times\mathbb{E}^1$-manifolds.) 
In \S7 we consider bundles with fibre a hyperbolic closed surface,
and we show that if a bundle space $Y$ is dominated by a product
$B\times{F}$ then it is also a product.
Most bundle spaces are not geometric.
In the final section we consider dominations of aspherical geometric 4-manifolds 
by bundles with hyperbolic fibre.
These include reducible $\mathbb{H}^2\times\mathbb{H}^2$-manifolds,
which are finitely covered by products, 
and $\mathbb{H}^3\times\mathbb{E}^1$-manifolds.
We have little to say about irreducible 
$\mathbb{H}^2\times\mathbb{H}^2$-manifolds,
or about the geometries $\mathbb{H}^4$
and $\mathbb{H}^2(\mathbb{C})$.

Our main interest is in the existence of maps of non-zero degree.
For this purpose we may replace the domain $X$ and and range $Y$ 
by more convenient finite covering spaces, 
and assume that the map induces an epimorphism on the fundamental group.
In several places we ask whether a specific 4-manifold $Y$ has such 
a domination by one of the representative geometric 4-manifolds 
under consideration.
Note that although geometric 4-manifolds have natural smooth structures, 
the use of topological surgery implicit in some of our arguments only justifies
identifications  of the possible range spaces up to homeomorphism.

This note was prompted by a query from R. \.I. Baykur,
arising from \cite{ABCKLZ}.
In that paper the authors consider the more specific question of
which closed 4-manifolds have branched coverings by the total spaces
of surface bundles.
Their main results are that every 1-connected closed 4-manifold
has a branched covering of degree $\leq16$ by a product
$B\times{T}$, with $B$ a closed surface and $T$ the torus,
and every product $B\times{S^1}$ with $B$ a closed orientable 3-manifold
has a 2-fold branched cover by a symplectic 4-manifold which fibres
over $T$.
The maps considered below often have degree $1$, 
and then are either homotopy equivalences or not 
homotopic to branched covers.

I would like to thank R. \.I. Baykur for his question and for his comments 
on an early draft, 
and S. Vidussi for pointing out a blunder in \S7 of the original arXival submission.

\section{dominations between aspherical geometric 4-manifolds}

In the aspherical case the underlying question is essentially one of group theory.
In \cite{Hi04} it is shown that $PD_n$-groups with {\it max}-$c$
may be partitioned into families, analogous to those of Wang,
and that the pattern of possible maps of non-zero degree is very similar.
A group has {\it max}-$c$ if all chains of centralizers in the group are finite.
A $PD_n$-pair of groups $(G,\partial{G})$ is {\it atoroidal\/} if 
every polycyclic subgroup of Hirsch length $n-1$ is conjugate into 
a boundary component,
and is of {\it Seifert type\/} if it has a normal polycyclic subgroup 
of Hirsch length $n-2$.
Kropholler showed that  all $PD_n$-groups with {\it max}-$c$
have JSJ decompositions along virtually polycyclic subgroups 
of Hirsch length $n-1$ into pieces which are either atoroidal or of Seifert type
 \cite{Kr90}.

When $n=4$ the qualification ``of Seifert type" reduces to 
``having a normal $\mathbb{Z}^2$ subgroup",
and these families of groups are either
\begin{enumerate}
\item{} atoroidal; 
\item{} have a non-trivial JSJ decomposition with at least one atoroidal piece; 
\item{} have a non-trivial JSJ decomposition with all pieces of Seifert type;
\item{}(a) virtually polycyclic, but not virtually of Seifert type; or
(b) virtually polycyclic and of Seifert type but not virtually nilpotent:
\item{} virtually a product $G\times\mathbb{Z}^2$ with $G$ a $PD_2$-group and $\chi(G)<0$;
\item{} Seifert type but not virtually a product nor virtually polycyclic;
\item{} virtually nilpotent but not virtually abelian; or
\item{} virtually abelian.
\end{enumerate}
We have preserved Wang's enumeration,
but in higher dimensions it is useful to subdivide  type (4), 
the analogue of the class of $\mathbb{S}ol^3$-manifolds.

The fundamental groups of aspherical $n$-manifolds are $PD_n$-groups,
and the groups of geometric $4$-manifolds satisfy {\it max}-$c$.
Not all aspherical 4-manifolds with groups of types (1) or (6) are geometric,
and there are no geometric 4-manifolds with groups of type (2) or (3).
The correspondence with geometries is 
(1) $\mathbb{H}^4$, $\mathbb{H}^2(\mathbb{C})$, 
$\mathbb{H}^2\times\mathbb{H}^2$ and
$\mathbb{H}^3\times\mathbb{E}^1$;
(4.a) $\mathbb{S}ol_{m,n}^4$ (with $m\not=n$) and $\mathbb{S}ol_1^4$; 
(4.b) $\mathbb{S}ol^3\times{E}^1$;
(5) $\mathbb{H}^2\times\mathbb{E}^2$;
(6) $\widetilde{\mathbb{SL}}\times\mathbb{E}^1$;
(7) $\mathbb{N}il^3\times\mathbb{E}^1$ and $\mathbb{N}il^4$; and 
(8) $\mathbb{E}^4$.

For geometric 4-manifolds in the families (4--8) the conclusion of
\cite{Hi04} is that all maps between groups of different types have degree 0,
except for maps from  groups of type (5) to groups of type (8) and 
from (6) to $\mathbb{N}il^3\times\mathbb{E}^1$-groups in type (7).
(The assertion there that there are maps of nonzero degree
from groups of type (5) to groups of type (4c) is wrong.)
On the other hand every $\mathbb{N}il^3\times\mathbb{E}^1$-group
and every $\mathbb{E}^4$-group is so dominated.
Theorem 1.1  of \cite{Ne} is slightly sharper,
in that it shows that there are no such maps between groups of distinct 
geometries within types (4.a) and in (7).

\section{some general observations}

If $X$ and $Y$ are closed orientable $n$-manifolds then 
$X$ {\it dominates}  $Y$ if there is a map $f:X\to{Y}$ 
with nonzero degree.
If so, then 
\begin{enumerate}
\item{} $d$-fold (branched) finite covers have degree $d$;
\item{}the image of $\pi_1(f)$ has finite index in $\pi_1(Y)$, 
so $f$ factors through a map $\widehat{f}:X\to\widehat{Y}$, 
where $\widehat{Y}$ covers $Y$, 
and $\pi_1(\widehat{f})$ is an epimorphism;
\item{}if $R$ is a ring in which $d=deg(f)$ is invertible then 
$H^*(Y;R)$ is a subring of $H^*(X;R)$,
and is a direct summand as an $R$-module:
\item{}hence if $n=4$ then $\chi(Y)\leq2+\beta_2(X;\mathbb{Q})$.
\end{enumerate}

If $f$ has degree $d$ (for some choice of orientations)
then we shall say that $X$ {\it $d$-dominates}  $Y$.
If $X$ {\it 1-dominates} $Y$ then $\pi_1(f)$ is an epimorphism,
and $H^*(Y;\mathbb{Z})$ is a direct summand of $H^*(X;\mathbb{Z})$.
We say that $X$ {\it essentially dominates\/} $Y$ if $f$ 
has non-zero degree and $\pi_1(f)$ is an epimorphism.
(Such maps need not have degree 1,
as is already clear when $X=S^2$.
Self maps of $S^2$ of degree $>1$ induce isomorphisms on $\pi_1$,
but do not induce splittings of $H^2(S^2;\mathbb{Z})$.)
Clearly $X$ dominates $Y$ if and only if $X$ essentially dominates 
some finite cover of $Y$.

We note also that if $X$ and $V$ are orientable closed 4-manifolds 
with $\chi(X)=\chi(Y)$ then a map  $f:X\to{Y}$ is a homotopy equivalence 
if and only if it has degree 1 and $\pi_1(f)$ is an isomorphism
\cite[Theorem 3.2]{Hi}.

\begin{lemma}
\label{tf}
Let $f:X\to{Y}$ be a degree-$1$ map  between orientable closed $4$-manifolds.
If $H_1(X;\mathbb{Z})$ is torsion-free then so is $H_1(Y;\mathbb{Z})$.
\end{lemma}

\begin{proof}
This follows from the Universal Coefficient Theorem,
since torsion in $H_1(Y;\mathbb{Z})$ is detected by torsion in 
$H^2(Y;\mathbb{Z})$,
and
$H^2(Y;\mathbb{Z})$ is a direct summand of $H^2(X;\mathbb{Z})$.
\end{proof}

\begin{lemma}
\label{Euler}
Let $f:X\to{Y}$ be a map of non-zero degree between orientable closed 
$4$-manifolds.
Suppose that there is an integer $D\geq0$ such that
$\beta_2(\hat{X};\mathbb{Q})\leq {D}$, 
for all finite covering spaces $\hat{X}$ of $X$.
Then
\begin{enumerate}
\item{}
if $\pi_1(Y)$ is infinite and has subgroups of arbitrarily large finite index 
then $\chi(Y)\leq0$;
\item{}if $\pi_1(Y)$ is finite then $|\pi_1(Y)|\chi(Y)\leq{D+2}$,
and so $|\pi_1(Y)|\leq\frac12(D+2)$.
\end{enumerate}
\end{lemma}

\begin{proof}
If $:\hat{Y}\to{Y}$ is a finite covering and $:\hat{X}\to{X}$ is the induced covering 
then $f$ lifts to a dominating map $\hat{f}:\hat{X}\to\hat{Y}$. 
On the one hand $\chi(\hat{Y})=[\pi_1(Y):\pi_1(\hat{Y})]\chi(Y)$;
on the other, $\chi(\hat{Y})\leq{D+2}$. 
The first assertion follows easily.

The second assertion has a similar proof.
(Note that $\chi(Y)\geq2$,
since $Y$ is orientable and $\pi_1(Y)$ is finite.)
\end{proof}

In conjunction with this lemma, note that if the $L^2$-Betti numbers
$\beta_i^{(2)}(Y)=0$ for $i\leq1$ then 
$\chi(Y)=\beta_2^{(2)}(Y)\geq0$, 
by the $L^2$-Euler characteristic formula.
(This is the case if $\pi_1(Y)$ is infinite and amenable,
or has a finitely generated infinite normal subgroup of infinite index.) 
 
\begin{lemma}
\label{w2}
Let $f:X\to{Y}$ be a map between orientable closed $4$-manifolds.
Then
\begin{enumerate}
\item{}if $f$ has odd degree  and $w_2(X)=0$ then $w_2(Y)=0$;
\item{}if $Y$ is $1$-connected and $w_2(Y)=0$ then $\chi(Y)$ is even.
\end{enumerate}
\end{lemma}

\begin{proof}
The map $f$ induces a monomorphism 
$H^*(f)$ from $H^*(Y;\mathbb{Z}/2\mathbb{Z})$
to $H^*(X;\mathbb{Z}/2\mathbb{Z})$.
Since $\xi^2=0$ for all $\xi\in{H^*(X;\mathbb{Z}/2\mathbb{Z})}$ 
the same is true for $H^*(Y;\mathbb{Z}/2\mathbb{Z})$, and so $w_2(Y)=0$.

If $M$ is an orientable 4-manifold then $w_2(M)^2=w_4(M)$, by the Wu formulae,
and so $[M]\cap{w_2(M)^2}\equiv\chi(M)$ {\it mod} (2).
Hence if $Y$ is 1-connected then $\chi(Y)$ is even.
\end{proof}

On the other hand, there is a degree-1 map from $CP^2$ to $S^4$,
and so $w_2(Y)=0$ does not imply that $w_2(X)=0$.

If $X$ is a cell complex of dimension $\leq4$ then $[X,CP^2]=[X,K(\mathbb{Z},2)]$,
by general position, since we may construct $K(\mathbb{Z},2)\simeq{CP^\infty}$ 
by adding cells of dimension $\geq6$ to $CP^2$.
Hence if $u$ is  a generator of $H^2(CP^2;\mathbb{Z})$ then
$f\mapsto{f^*u}$ defines a bijection $[X,CP^2]\to{H^2(X;\mathbb{Z})}$. 
If $X$ is a closed orientable 4-manifold the degree of $f$ is given by
$d=[X]\cap(f^*u)^2$.

An element $\xi\in{H^2(X;\mathbb{Z})}$ is in the image of
$[X,S^2]=[X,CP^1]$ if and only if $\xi^2=0$ \cite[Theorem 8.11]{Sp}.

There is a similarly defined surjection from $[X,S^3]$ to $H^3(X;\mathbb{Z})$.
 
\begin{lemma}
\label{homology rank}
Let $X=M\times{S^1}$, where $M$ is a $3$-manifold,
and let  $f:X\to{Y}$ be an essentially dominating map.
If the image of the $S^1$-factor in $\pi_Y$ is infinite then  
$Y\simeq{P\times{S^1}}$,  where $P$ is a $PD_3$-complex.
\end{lemma}

\begin{proof}
The hypotheses imply that $\pi_Y$ is a product $\sigma\times\mathbb{Z}$, 
for some group $\sigma$.
We may apply the argument of Lemma \ref{Euler}
to the covering spaces $X_n$ and $Y_n$ associated to the subgroups 
of the form $\pi_1(M)\times{n\mathbb{Z}}$ and $\sigma\times{n\mathbb{Z}}$
to show that $\chi(Y)\leq0$.
On the other hand,
\[
\beta_2(Y_n)\geq\beta_2(\sigma\times\mathbb{Z})\geq\beta_1(\sigma)=\beta_1(Y_n)-1,
\]
and so $n\chi(Y)=\chi(Y_n)\geq-\beta_1(\sigma)-1$ for all $n$.
Hence $\chi(Y)=0$, and so $Y\simeq{P}\times{S^1}$, 
where $P$ is a $PD_3$-complex \cite[Theorem 4.5]{Hi}.
\end{proof}

\begin{lemma}
\label{mapping torus}
Let $M$ be the mapping torus of a self-homeomorphism
$\varphi$  of an $n$-manifold $N$. 
Then $M$ $1$-dominates ${S^n}\times{S^1}$.
\end{lemma}

\begin{proof}
We may assume that $\varphi$ fixes a disc $D^n\subset{N}$.
Collapsing the image of $\overline{N\setminus{D^n}}$ to a point in each fibre
induces a map from $M$ to $S^n\times{S^1}$ which clearly has degree 1.
\end{proof}

The following simple lemma is based on \cite[Lemma 2.3]{HLWZ}.

\begin{lemma}
\label{lens space}
Let $P$ be a $PD_3$-complex with finite fundamental group $\rho$.
If $M$ is an orientable $3$-manifold such that $\pi_1(M)$ maps onto $\rho$
then there is a map $f:M\to{P}$ with non-zero degree. 
\end{lemma}

\begin{proof}
Since $\rho$ is finite,  $P$ is orientable and has universal cover
$\widetilde{P}\simeq{S^3}$.
Hence $\pi_2(P)=0$, and so we may construct a $K(\rho,1)$ space 
$K=P\cup{e^{\geq4}}$ 
by adding cells of dimension $\geq4$ to $P$.
Let $g:M\to{K}$ be a map such that $\pi_1(g)$ is an epimorphism.
We may assume that $g$ has image in $P$, by cellular approximation.
If $g:M\to{P}$ has non-zero degree then we set $f=g$.
Otherwise,
let $p:M\to{M\vee{S^3}}$ be the pinch map,
and let $c:S^3\to{P}$ be the composition of a homotopy equivalence
$S^3\simeq\widetilde{P}$ with the universal covering projection. 
Then $f=(g\vee{c})\circ{p}$ has the desired properties.
\end{proof}

When $\rho$ is cyclic and $H^1(M;\mathbb{Z})\not=0$
we may assume that $g$ factors through $S^1$,
and so has degree 0. 
In this case $f$ has degree $|\rho|$, 
as in \cite{HLWZ}.

\begin{lemma}
\label{asph}
Let $f:X\to{Y}$ be a map between orientable closed $4$-manifolds.
If $Y$ is aspherical and $\pi_1(f)$ factors through a group $G$ 
such that $H_4(G;\mathbb{Q})=0$ then $f$ has degree $0$.
\end{lemma}

\begin{proof}
If $Y$ is aspherical then $f$ is determined by $\pi_1(f)$,
and so factors through $K(G,1)$.
Since $H_4(G;\mathbb{Q})=0$ the lemma follows.
\end{proof}

We shall henceforth assume that all manifolds considered are closed 
and (excepting $RP^2$) orientable,
and that $X$ and $Y$ are 4-manifolds,
$\pi_1(f)$ is an epimorphism and $f:X\to{Y}$ has degree $d\not=0$.
All homology and cohomology groups have coefficients $\mathbb{Q}$, 
unless otherwise specified.
If $F$ is a subgroup of $\pi_1(Y)$ then $Y_F$ is the associated covering space.
If $W$ and $Z$ are topological spaces then we write $W\simeq{Z}$ 
if they are homotopy equivalent and $W\cong{Z}$ if they are homeomorphic.

\section{compact or mixed compact-aspherical}

Suppose that $X$ has one of the compact or mixed compact-solvable
geometries $\mathbb{S}^4$, 
$\mathbb{CP}^2$, $\mathbb{S}^2\times\mathbb{S}^2$, 
$\mathbb{S}^3\times\mathbb{E}^1$ or $\mathbb{S}^2\times\mathbb{E}^2$.
Then $X$ is finitely covered by one of $S^4$, 
$CP^2$, $S^2\times{S^2}$, 
$S^3\times{S^1}$ or $S^2\times{T}$, respectively.
With these geometries we shall also consider the bundle space
$S^2\tilde\times{S^2}$ and the mixed compact-aspherical geometry 
$\mathbb{S}^2\times\mathbb{H}^2$.
(See \cite[Chapters 10--12]{Hi}.)

\medskip
$\mathbb{S}^4$. 
We may assume that $X=S^4$.
Then $\pi_1(Y)=1$ and $\beta_2(Y)=0$, 
and so $Y\simeq{S^4}$.
If $d=1$ then $f$ is homotopic to a homeomorphism.

\medskip
$\mathbb{CP}^2$. 
We may assume that $X=CP^2$.
Then $\pi_1(Y)=1$ and $\beta_2(Y)=1$ or  0.
Hence $Y$ is homeomorphic to one of $CP^2$, $Ch=*CP^2$ 
(the fake complex projective plane) or $S^4$.

\medskip
$\mathbb{S}^2\times\mathbb{S}^2$. 
We may assume that $X=S^2\times{S^2}$.
Then $\pi_1(Y)=1$ and $\beta_2(Y)=2,1$ or 0.
If $\beta_2(Y)=2$ then $Y\cong{X}$,
$S^2\tilde\times{S^2}=CP^2\sharp-CP^2$ or $CP^2\sharp{CP^2}$.
If also $d=1$ then $f$ is homotopic to a homeomorphism.
Let $q:S^2\tilde\times{S^2}\to{S^2}$ be the nontrivial $S^2$-bundle 
over $S^2$.
The pullback of $q$ over a degree-2 map from $S^2$ to $S^2$
is a trivial bundle.
Hence there is a degree-2 map from $X$ to $S^2\tilde\times{S^2}$.
(There is no map of degree 1 \cite[Theorem 3.2]{Hi},
and the degree must be even, by Lemma \ref{w2}.)

If $\beta_2(Y)=1$ then $Y\simeq{CP^2}$. 
There are maps $f:S^2\times{S^2}\to{CP^2}$ of every even degree,
but none of degree 1.

If $\beta_2(Y)=0$ then $Y$ is homeomorphic to $S^4$.

\medskip
$S^2\tilde\times{S^2}$. 
There are degree-1 maps from $S^2\tilde\times{S^2}$ to 
$CP^2$ and to $S^4$,
and there is a degree-2 map to $S^2\times{S^2}$.
(There is no map of degree 1 \cite[Theorem 3.2]{Hi}.)

Maps between $S^2\times{S^2}$ or $S^2\tilde\times{S}^2$ 
and $CP^2\sharp{CP^2}$ have degree 0, since the former have
signature 0 while $CP^2\sharp{CP^2}$ has signature 2,
and a map $f$ of non-zero degree would induce a ring isomorphism
$H^*(f;\mathbb{Q})$.

\medskip
$\mathbb{S}^3\times\mathbb{E}^1$. 
We may assume that $X=S^3\times{S^1}$.
Then $\pi_1(Y)$ is cyclic, and $\beta_2(Y)=0$.

If $\pi_1(Y)\cong\mathbb{Z}$ then $Y\cong{S^3\times{S^1}}$
\cite[Theorem 11.1]{Hi}. 
If also $d=1$ then $f$ is homotopic to a homeomorphism.

If $\pi_1(Y)\cong\mathbb{Z}/n\mathbb{Z}$ then $n=1$,
by Lemma \ref{Euler}.
Hence $Y$ is homeomorphic to $S^4$.

\medskip
$\mathbb{S}^2\times\mathbb{E}^2$.
We may assume that $X=S^2\times{T}$.
If $\pi_1(Y)$ is infinite then $\chi(Y)=0$, 
by Lemma \ref{Euler} and the subsequent remark.

If $\beta_1(Y)=2$ then $\pi_1(Y)\cong\mathbb{Z}^2$,
and so $\pi_1(f)$ is an isomorphism.
Since $H^2(Y;\mathbb{Z})$ is a direct summand of $H^2(X;\mathbb{Z})$,
it follows that $H^*(f)$ is an isomorphism of cohomology rings.
In particular,  $f$ has degree 1.
Since $\chi(Y)=\chi(X)$, these observations imply that $f$ is a homotopy equivalence \cite[Theorem 3.2]{Hi}.
Hence $f$ is homotopic to a homeomorphism \cite[Theorem 6.16]{Hi}.

If $\beta_1(Y)=1$ then $\chi(Y)=0$ and 
$\pi_1(Y)\cong\mathbb{Z}\oplus\mathbb{Z}/n\mathbb{Z}$, for some $n\geq1$.
Hence $Y\simeq{L}\times{S^1}$, 
where $L$ is a lens space \cite[Theorem 11.1]{Hi},
and so $S^2\times{T}=(S^2\times{S^1})\times{S^1}$ essentially dominates $Y$, 
by Lemma \ref{lens space}.

If $\beta_1(Y)=0$ then $|\pi_1(Y)|\chi(Y)\leq4$, 
by Lemma \ref{Euler}.
There are obvious degree-1 maps to $S^2\times{S^2}$
(since $T$ 1-dominates $S^2$) and $S^4$.
There are degree-2 maps to $S^2\tilde\times{S^2}$ and to $CP^2$,
but none of degree 1.

If $\chi(Y)=2$ and $\pi_1(Y)=\mathbb{Z}/2\mathbb{Z}$ then $Y$ is homotopy equivalent to one of the two orientable 4-manifolds which are total spaces of $S^2$-bundles over $RP^2$ \cite[Chapter 12]{Hi}.
A map $f:X\to{Y}$ of non-zero degree lifts to a map
$f^+:X^+\to\widetilde{Y}=S^2\times{S^2}$ 
of non-zero degree, which induces an isomorphism
$H_2(f^+;\mathbb{Q}):H_2(X^+;\mathbb{Q})\to
{H_2(\widetilde{Y};\mathbb{Q})}=\mathbb{Q}\otimes\pi_2(Y)$.
Since $\pi_1(X$ acts trivially on $H_2(X^+;\mathbb{Q})$ 
but $\pi_1(Y)$ acts non-trivially on $\pi_2(Y)$ it follows that
$\pi_1(f)$ cannot be an epimorphism.
Thus $S^2\times{T}$ does not essentially dominate either of these bundle spaces.

\medskip
$\mathbb{S}^2\times\mathbb{H}^2$.
If $X$ is a $\mathbb{S}^2\times\mathbb{H}^2$-manifold
then it is finitely covered by $S^2\times{B}$,
where $B$ is a hyperbolic surface.
Since there is a degree-1 map from $B$ to $T$, 
there is a degree-1 map from $S^2\times{B}$ to $S^2\times{T}$, 
and hence there are such maps to $S^3\times{S^1}$, 
$S^2\times{S^2}$ and $S^4$.
Thus every $\mathbb{S}^2\times\mathbb{H}^2$-manifold and every 4-manifold
with one of the five compact or mixed compact-solvable geometries above
is dominated by such a product.
On the other hand, 
any map from $S^2\times\hat{B}$ to an aspherical 4-manifold factors 
through the projection to $\hat{B}$, and so has degree 0.

\section{solvable lie type}

There are six solvable Lie geometries.
One is an infinite family $\mathbb{S}ol_{m,n}^4$ of closely related geometries, 
which includes the product geometry
$\mathbb{S}ol^3\times\mathbb{E}^1$ as the equal parameter case $m=n$.
This product geometry needs separate consideration here.
(See \cite[Chapters 7--8]{Hi} for details of these geometries 
and the associated fundamental groups.
It appears to be unknown when different
pairs $(m,n)$ and $(m',n')$
determine the same geometry $\mathbb{S}ol_{m,n}^4$. 
See \cite[page 137]{Hi}.)

Suppose that $X$ has a solvable Lie geometry.
Then $\pi_X=\pi_1(X)$ is polycyclic of Hirsch length 4
and $\chi(X)=0$, so $\beta_2(X)\leq6$.
After passing to a finite covering space, if necessary,
we may assume that $X$ is a solvmanifold
(i.e., a coset space of a 1-connected solvable Lie group )
and either $X=T^4$ or $\pi_X$ is not virtually abelian.
Thus $X$ is parallelizable.
Moreover, 
$X$ is a mapping torus $N\rtimes{S^1}$, 
where $N=T^3$ or is a coset space of $Nil^3$.
If $\pi_X$ is nilpotent and $S=\{g_1,\dots,g_\beta\}$ represents a basis for the 
maximal torsion free quotient of $\pi_X/\pi_X'$ then the subgroup
generated by $S$ has finite index in $\pi_X$ and $S/S'\cong\mathbb{Z}^\beta$.
Hence in this case we may also assume that 
$H_1(X;\mathbb{Z})$ is torsion-free.

\begin{lemma}
Let $f:X\to{Y}$ be a map of non-zero degree between orientable closed 
$4$-manifolds.
If $X$ is a solvmanifold and $\pi_1(f)$ is an epimorphism then
\begin{enumerate}
\item$\pi_1(Y)$ is polycyclic of Hirsch length $h_Y\leq4$,
and $w_2(Y)=0$;
\item{}
if $h_Y=0$ then $\pi_1(Y)$ is finite and $2\leq\chi(Y)\leq2+\beta_2(X)$;
\item{}if $h_Y=1$ and $F$ is the maximal finite normal subgroup of $\pi_1(Y)$
then $Y_F$ is a $PD_3$-complex,
and $Y$ is finitely covered by $S^3\times{S^1}$;
\item{}
if $h_Y=2$ then $Y$ is finitely covered by $S^2\times{T}$;
\item{}
if $h_Y>2$ then $f$ is homotopic to a homeomorphism.
\end{enumerate}
\end{lemma}

\begin{proof}
Since $\pi_Y=\pi_1(Y)$ is a quotient of $\pi_X$, 
it is polycyclic and $h_Y\leq4$.
Since $X$ is parallelizable, $w_2(X)=0$ and so $w_2(Y)=0$, 
by Lemma \ref{w2}.

It follows from Lemma \ref{Euler} and the subsequent remark that 
if $h_Y=0$ then $2\leq\chi(Y)\leq2+\beta_2(X)$,
while if $h_Y>0$ then $\chi(Y)=0$.

Part (3) is taken from \cite[Theorem 11.1]{Hi}.

If $h_Y=2$ then $Y$ has a covering space of degree dividing 4 
which is homeomorphic to $S^2\times{T}$ \cite[Theorem 10.1]{Hi}.

If $h_Y>2$ then $H^i(\pi_Y;\mathbb{Z}[\pi_Y])=0$ for $i\leq2$.
Since we also have $\chi(Y)=0$, 
it follows that $Y$ is aspherical 
and $h_Y=4$. 
In this case $\pi_1(f)$ is an isomorphism,
and so $f$ is homotopic to a homeomorphism \cite[Theorem 8.1]{Hi}.
\end{proof}

In particular, there are no maps of non-zero degree between manifolds 
having distinct solvable Lie geometries.

\begin{cor}
\label{nilcor}
If $\pi_Y$ is nilpotent,  $h_Y=2$, $H_1(Y;\mathbb{Z})$ is torsion-free 
and $w_2(Y)=0$ then $Y\cong{S^2}\times{T}$.
\end{cor}

\begin{proof}
Since $Y$ has a finite cover homeomorphic to $S^2\times{T}$,
the other hypotheses imply that $\pi_Y\cong\mathbb{Z}^2$ 
\cite[Theorem 10.14]{Hi}.
Therefore $Y\cong{S^2}\times{T}$,
since $Y$ is orientable and $w_2(Y)=0$.
\end{proof}

We shall consider the possibilities for $X$ in decreasing order of
$\beta_1(X)$.  

\medskip
$\mathbb{E}^4$. 
Every flat 4-manifold is finitely covered by $T^4$,
so we may assume that $X=T^4$. 
Then $\pi_Y$ is abelian,  $\beta_1(Y)\leq4$ and $\beta_2(Y)\leq6$.

If $f$ has degree 1 then $\pi_Y$ is torsion-free, by Lemma \ref{tf},
and so $\pi_Y\cong\mathbb{Z}^r$, where $r=4,2,1$ or 0.
There are obvious degree-1 maps to $T^4$, $S^2\times{T}$, 
$S^3\times{S^1}$, $S^2\times{S^2}$ and $S^4$.
The first three are the only possibilities for $Y$ with $\pi_Y$ infinite,
since $w_2(Y)=0$, by Lemma \ref{w2}.
If $\pi_Y=1$ then $\chi(Y)=2$, 4, 6 or 8,
by Lemmas \ref {Euler} and \ref{w2}.

Suppose now that $f$ is essentially dominating, but has degree $>1$.
If $h_Y=2$ then $\pi_Y\cong\mathbb{Z}^2$,
by \cite[Theorems 10.14 and 10.16]{Hi}, 
since $Y$ is orientable and $\pi_Y$ is abelian.
Thus $Y$ is the total space of an $S^2$-bundle over $T$
\cite[Theorem 10.10]{Hi}.

If $h_Y=1$ then $\pi_Y\cong\mathbb{Z}\oplus{A}$,
where $A$ is a finite abelian group.
Hence $Y\simeq{L\times{S^1}}$, for some lens space $L$ with
$\pi_1(L)\cong\mathbb{Z}/n\mathbb{Z}$.
Now $L\times{S^1}$ is essentially $n$-dominated by $S^2\times{T}$
(as in \S3 above),
and so there is an essentially dominating map $f:X\to{L\times{S^1}}$ 
of degree $n$.

If $h_Y=0$ then $|\pi_1(Y)|\chi(Y)\leq8$,
by Lemma \ref{Euler}.
Suppose that $f:X\to{Y}$ has nonzero degree,
and let $X^+\to{X}$ be the covering induced from the universal cover 
$\widetilde{Y}\to{Y}$.
The lift $f^+:X^+\to\widetilde{Y}$ induces an epimorphism 
$H_2(f^+;\mathbb{Q})$.
Since $\beta_2(X^+)=\beta_2(X)$,
the covering projection induces an isomorphism
$H_2(X^+;\mathbb{Q})\cong{H_2(X;\mathbb{Q})}$,
and so $\pi_X$ acts trivially on $H_2(X^+;\mathbb{Q})$.
But if $\pi_Y\not=1$ then $\pi_Y$ acts non-trivially on
$H_2(\widetilde{Y};\mathbb{Q})$,
and so $\pi_1(f)$ cannot be an epimorphism.
Thus we may assume that $\pi_Y=1$.
Since there are degree-1 maps from $X$ to $S^2\times{T}$,
there are degree-2 maps from $X$ to $CP^2$ and to 
$S^2\tilde\times{S^2}$.

We have no examples with $\pi_1(Y)=1$ and $5\leq\chi(Y)\leq8$.

\medskip
$\mathbb{N}il^3\times\mathbb{E}^1$.
Every closed $\mathbb{N}il^3\times\mathbb{E}^1$-manifold is 
finitely covered by $N\times{S^1}$, 
where $N$ is the total space of the $S^1$ bundle over the torus $T$ 
with Euler class the generator of $H^2(T;\mathbb{Z})$,
so we may assume that $X=N\times{S^1}$ and 
$H_1(X;\mathbb{Z})\cong\mathbb{Z}^3$.
Moreover, if $\hat{X}$ is a finite covering space of $X$ then 
$\beta_2(\hat{X};\mathbb{Q})=4$.
In this case $\pi_Y$ is nilpotent, $\beta_1(Y)\leq3$ and $\beta_2(Y)\leq4$.

If $f$ has degree 1 then $H_1(Y;\mathbb{Z})$ is torsion-free, 
by Lemma \ref{tf}.
Hence if $h_Y=2$ then $Y\cong{S^2}\times{T}$,  by Corollary \ref{nilcor}.
If $h_Y=1$ then $\pi_Y\cong\mathbb{Z}$ and so $Y\cong{S^3}\times{S^1}$
\cite[Theorem 11.1]{Hi}.
If $h_Y=0$ then $\pi_Y=1$ and $\chi(Y)=2$, 4 or 6, and $w_2(Y)=0$,
by Lemmas \ref {Euler} and \ref{w2}.
Since $N$ 1-dominates $S^2\times{S^1}$, 
by Lemma \ref{mapping torus},
there is a degree-1 map from $X$ to $S^2\times{T}$.
Hence there also degree-1 maps to $S^3\times{S^1}$,
$S^2\times{S^2}$ and $S^4$.

Suppose now that $f$ is essentially dominating, but has degree $>1$.
If $h_Y=2$ then $Y$ is the total space of an $S^2$-bundle over $T$
(as for $\mathbb{E}^4$).

If $h_Y=1$ then the maximal finite normal subgroup 
$F$ of $\pi_Y$ is nilpotent,
and $Y$ is homotopy equivalent to the mapping torus of a self-homotopy equivalence of $Y_F$.
In particular, if the image of the $S^1$ factor of $X$ 
has infinite order in $\pi_Y$ then $F$ is a quotient of $\pi_1(N)$ and
$Y\simeq{Y_F}\times{S^1}$.
Since the commutator subgroup $\pi_1(N)'$ is central,
$F'$ is central, 
and so $F$ is cyclic or  is the product of $Q(8)$ with a cyclic group of odd order.
(Thus $F$ is a 3-manifold group.)
There are essentially dominating maps $f:X\to{Y}$,
by Lemma \ref{lens space}.

On the other hand, if the image of the $S^1$ factor is finite
then the image of $\pi_1(N)$ in $\pi_1(Y)$ is infinite.
Hence $F$ is abelian, 
and so $Y_F$ is homotopy equivalent to a lens space.
We do not know whether there are examples of this type 
(other than when $Y$ is a product).

If $h_Y=0$ then $|\pi_1(Y)|\chi(Y)\leq6$, 
by Lemma \ref{Euler}.
As in the case $\mathbb{E}^4$, we may assume that $\pi_Y=1$,
and there are again degree-2 maps from $X$ to $CP^2$ and to
$S^2\tilde\times{S^2}$.
We have no examples with $\chi(Y)=5$ or 6,

\medskip
$\mathbb{N}il^4$.
There is again a canonical choice for $X$.
There is an unique torsion-free nilpotent group of Hirsch length 4 
which can be generated by 2 elements \cite[Corollary 11]{Hi21}.
This group has the presentation
\[
\langle{t,u}\mid[t,[t,[t,u]]]=[u,[t,u]]=1\rangle.
\]
We may assume that $X$ is the corresponding $\mathbb{N}il^4$-manifold,
which is the mapping torus of a self-homeomorphism of $T^3$,
and so $H_1(X;\mathbb{Z})\cong\mathbb{Z}^2$.
Moreover, if $\hat{X}$ is a finite covering space of $X$ then 
$\beta_2(\hat{X};\mathbb{Q})=2$.
In this case $\pi_Y$ is nilpotent, $\beta_1(Y)\leq2$ and $\beta_2(Y)\leq2$.

If $f$ has degree 1 and $h_Y>0$ then $Y\cong{S^2}\times{T}$ 
or $S^3\times{S^1}$ (as in the $\mathbb{N}il^3\times\mathbb{E}^1$ case).
There are degree-1 maps to $S^2\times{T}$ and to $S^3\times{S^1}$,
and hence also to $S^2\times{S^2}$ and $S^4$.
If  $h_Y=0$ then $\pi_Y=1$ and $\chi(Y)=2$ or 4, 
and $w_2(Y)=0$,  by Lemmas \ref {Euler} and \ref{w2}.
Hence $Y\cong{S^4}$ or $S^2\times{S^2}$.

Suppose now that $f$ is essentially dominating, but has degree $>1$.
If $h_Y=2$ then $Y$ is the total space of an $S^2$-bundle over $T$
(as for $\mathbb{E}^4$).

If $h_Y=1$ then the maximal finite normal subgroup $F$ of $\pi_Y$
is  nilpotent and $Y$ is homotopy equivalent to the mapping torus 
of a self-homotopy equivalence of $Y_F$.
Since the kernels of maps from $\pi_1(X)$ to $\mathbb{Z}$
have nilpotency class $\leq2$, 
either $F\cong{Q(8)}\times{C}$ with $C$ cyclic of odd order
or $F$ is cyclic.
We do not know whether there are examples of this type. 

If $h_Y=0$ then $|\pi_1(Y)|\chi(Y)\leq4$,
by Lemma \ref{Euler}.
As in the case $\mathbb{E}^4$, we may assume that $\pi_Y=1$,
and there are again degree-2 maps from $X$ to $CP^2$ and to
$S^2\tilde\times{S^2}$.

In the remaining cases there are no natural choices for $X$, 
and Corollary \ref{nilcor} does not apply.

\medskip
$\mathbb{S}ol^3\times\mathbb{E}^1$.
Every $\mathbb{S}ol^3\times\mathbb{E}^1$-manifold is finitely covered 
by a product $X=P\times{S^1}$, 
where $P$ is orientable and is the total space of a $T$-bundle over $S^1$,
and $\beta_1(X)=\beta_2(X)=2$.
(The latter conditions hold for all finite covering spaces $\hat{X}$ of $X$.)
The manifold $X$ is the mapping torus of a self-homeomorphism of $T^3$.

Since $P$ 1-dominates $S^2\times{S^1}$, by Lemma \ref{mapping torus},
there is a degree-1 map from $X$ to $S^2\times{T}$,
and hence also to $S^3\times{S^1}$, $S^2\times{S^2}$ and $S^4$.
 
Suppose now that $f$ is essentially dominating, but has degree $>1$.
If $h_Y=2$ then $\pi_Y$ maps onto $\mathbb{Z}^2$,
since the non-trivial normal subgroups of infinite index in $\pi_1(P)$
are commensurate with $\pi_1(P)'$.
We again find that $Y$ is the total space of an $S^2$-bundle over $T$
(as for $\mathbb{E}^4$).

If $h_Y=1$ and $F$ is the maximal finite normal subgroup of $\pi_Y$
then $F'$ is abelian (so $F$ is cyclic,
metacyclic or generalized quaternionic),
and $Y$ is homotopy equivalent to the mapping torus 
of a self-homotopy equivalence of $Y_F$.
In particular, if the image of the $S^1$ factor of $X$ 
has infinite order in $\pi_Y$ then $F$ is a quotient of $\pi_1(N)$ and
$Y\simeq{Y_F}\times{S^1}$.
There are essentially dominating maps $f:X\to{Y}$,
by Lemma \ref{lens space}.

On the other hand, if the image of the $S^1$ factor is finite
then the image of $\pi_1(N)$ in $\pi_1(Y)$ is infinite.
Hence $F$ is abelian, 
and so $Y_F$ is homotopy equivalent to a lens space.
We do not know whether there are examples of this type 
(other than when $Y$ is a product).

If $h_Y=0$ then $|\pi_1(Y)|\chi(Y)\leq4$,
by Lemma \ref{Euler}.
As in the case $\mathbb{E}^4$, we may assume that $\pi_Y=1$,
and there are again degree-2 maps from $X$ to $CP^2$ and to
$S^2\tilde\times{S^2}$.

\medskip
$\mathbb{S}ol^4_{m,n}$ with $m\not=n$,
$\mathbb{S}ol^4_0$ and $\mathbb{S}ol^4_1$.
In these cases we may assume that $X$ is the mapping torus  
of a self-homeomorphism of $T^3$ or of a closed $\mathbb{N}il^3$-manifold,
and that $\beta_1(X)=1$ and $\beta_2(X)=0$.
(The latter conditions hold for all finite covering spaces $\hat{X}$ of $X$.)
There are degree-1 maps from $X$ to $S^3\times{S^1}$
(by Lemma \ref{mapping torus}) and to $S^4$.

If $f$ is an essentially dominating map but is not homotopic to 
a homeomorphism then $h_Y<2$,
since no quotient of $\pi_X$ has Hirsch length 2.
If $h_Y=1$ and $F$ is the maximal finite normal subgroup of $\pi_Y$
then either $F'$ is abelian (if the geometry is 
$\mathbb{S}ol^4_{m,n}$ or $\mathbb{S}ol^4_0$),
or $F$ is cyclic or $Q(8)\times{C}$ with $C$ cyclic of odd order
(if the geometry is $\mathbb{S}ol^4_1$),
and $Y$ is homotopy equivalent to the mapping torus 
of a self-homotopy equivalence of $Y_F$.
We do not know whether there are examples of this type
(other than $Y=S^3\times{S^1}$).

If $h_Y=0$ then $\chi(Y)=2$ and $\pi_Y=1$, 
by Lemma \ref{Euler},
and so $Y\cong{S^4}$.

\section{$\mathbb{H}^2\times\mathbb{E}^2$,
$\widetilde{\mathbb{SL}}$ and $T$-bundles}

In this section we assume that the domain $X$ is the total space of a
bundle $p:X\to{B}$, with base a hyperbolic surface $B$ and fibre $T$.
(Bundles with hyperbolic base and fibre $S^2$ are 
$\mathbb{S}^2\times\mathbb{H}^2$-manifolds, 
and were considered in \S3 above.
The case with hyperbolic base and fibre is considered in \S7 below.)
For brevity, we may call $X$ a {\it $T$-bundle space}. 

Clearly $\chi(X)=\chi(B)\chi(F)=0$.
The bundle $p$ is determined by its {\it monodromy} 
$\theta:\pi_1(B)\to{GL(2,\mathbb{Z})}$ and a class $[p]\in{H^2(B;\mathcal{T})}$, 
where $\mathcal{T}$ is the $\pi_1(B)$-module determined by $\theta$.
Such a bundle has a section if and only if $[p]=0$.
The total space $X$ is geometric if and only if the monodromy has finite image.
If this is so and the bundle has a section then
$X$ is finitely covered by $\hat{B}\times{T}$, 
where $\hat{B}$ is a hyperbolic surface,.
In this case the geometry is $\mathbb{H}^2\times\mathbb{E}^2$.
The other possible geometry is $\widetilde{\mathbb{SL}}\times\mathbb{E}^1$,
and then $X$ is finitely covered by $M\times{S^1}$, 
where $M$ is a $\widetilde{\mathbb{SL}}$-manifold
\cite[Corollary 7.3.1 and Theorem 9.3]{Hi}.

If the monodromy is infinite then either it 
preserves a flag ${\mathbb{Z}<\mathbb{Z}^2}$,  
or its image contains a matrix whose eigenvalues are not roots of unity.
In the first case, after passing to a double covering space if necessary,
the total space $X$ is the total space of a principal $S^1$-bundle 
over a 3-manifold $M$, 
which is in turn the total space of a principal $S^1$-bundle over a surface.
If $p$ has a section then $M$ is a $\mathbb{H}^2\times\mathbb{E}^2$-manifold;
otherwise $M$ is a $\widetilde{\mathbb{SL}}$-manifold. 
If the monodromy does not preserve a flag then the extension class group 
$H^2(B;\mathcal{T})$ is finite.

Let $A$ be the image of $\pi_1(T)$ in $\pi_Y$.
Then either $\pi_1(f|_T)$ is injective,
or $A$ has rank 1 or it is finite.
If $A\cong\mathbb{Z}^2$, $\pi_Y/A$ is infinite and $\chi(Y)=0$
then $Y$ is aspherical \cite[Theorem 9.2]{Hi}.
Conversely,
if $Y$ is aspherical then $A\cong\mathbb{Z}^2$,
for otherwise $f$ would have degree 0, by Lemma \ref{asph}.
It then follows that $\pi_Y/A$ is virtually a $PD_2$-group 
\cite[Theorem 3.10]{Hi},
and so $\chi(Y)=\chi(A)\chi^{virt}(\pi_Y/A)=0$.
If $A$ has rank 1 and finite index in $\pi_Y$ then $Y$ is finitely covered 
by $S^3\times{S^1}$, by Lemma \ref{homology rank}.

If $X$ is not geometric then $\pi_X$ has no $\mathbb{Z}$ normal subgroup,
and so either $A\cong\mathbb{Z}^2$ or $A$  is finite.

In the next theorem we shall assume that $A\cong\mathbb{Z}^2$
and that $Y$ is dominated by a $T$-bundle space.
Having an infinite abelian normal subgroup implies that $\chi(Y)\geq0$
\cite[Theorem 3.4]{Hi},
but something more seems needed to ensure that $\chi(Y)=0$.

\begin{theorem}
\label{split} 
Let $X$ be the total space of a $T$-bundle $p:X\to{B}$
over a hyperbolic surface $B$,
and let  $f:X\to{Y}$ be a map with non-zero degree
and such that $\pi_1(f|_T)$ is injective.
Then
\begin{enumerate}
\item{} if $X\cong{M}\times{S^1}$ then $Y\simeq{P}\times{S^1}$,
where $P$ is a Seifert fibred $3$-manifold;
\item{}if $X\cong{B\times{T}}$ then $Y\simeq{C}\times{T}$, 
where $C$ is a surface;
\item{}if $p$ has a section but $X$ is not geometric then $Y$ is aspherical,
and has a finite covering which is homotopy equivalent to 
the total space of a $T$-bundle with a section.
\end{enumerate}
In each case, if $q:Y\to\overline{B}$ is  a $T$-bundle then the monodromy 
for $p$ maps isomorphically onto the monodromy for $q$, 
and $f$ is homotopic to the map induced by a map of bases.
\end{theorem}

\begin{proof}
If $X$ is geometric then we may assume that $\pi_X\cong\rho\times\mathbb{Z}$,
where $\rho\cong\pi_1(B)\times\mathbb{Z}$ or $\pi_1(M)$,
while if $p$ has a section then $\pi_X$ is a semidirect product 
$\mathbb{Z}^2\rtimes\beta$.
We may assume also that $f_*=\pi_1(f)$ is an epimorphism.
Let $A\cong\mathbb{Z}^2$ be the image of $\pi_1(T)$ in $\pi_Y$.
Then $\pi_Y\cong{f_*\rho}\times\mathbb{Z}$ is the product 
of two finitely presentable proper normal subgroups (if $X$ is geometric) 
and $\pi_Y$ is a semidirect product  $A\rtimes(\pi_Y/A)$ (if $p$ has a section).
In each case we may apply the argument of Lemmas \ref{Euler}
and \ref{homology rank} to the 
covering spaces associated to the subgroups 
of the form $\rho\times{n\mathbb{Z}}$ and $f_*\rho\times{n\mathbb{Z}}$,
and $n\mathbb{Z}^2\rtimes\beta$ and $nA\rtimes(\pi_Y/A)$,
respectively, to conclude that $\chi(Y)=0$.

It follows that if $\pi_Y/A$ is infinite then $Y$ is aspherical
\cite[Theorem 9.2]{Hi}.
If $\pi_Y/A$ is finite and $\chi(Y)=0$ then $Y$ is finitely covered by 
$S^2\times{T}$ \cite[Theorem 10.13]{Hi}.
We may then  assume that $A=\pi_Y$, 
and so the inclusion of $\pi_1(T)$ into $\pi_X$ splits.
Hence $\pi_X\cong\mathbb{Z}^2\times\pi_1(B)$.
Thus if $p$ has a section but $X$ is not geometric $Y$ must be aspherical.

The further detail in cases (1) and (2) is clear.
In case (3), let $G=\pi_Y/A$. 
Since  $A\cong\mathbb{Z}^2$ and $Y$ is aspherical, 
$G$ is virtually a $PD_2$-group \cite[Theorem 3.10]{Hi}.
Hence $Y$ is virtually homotopy equivalent to
the total space of a $T$-bundle $q:Y\to{C}$.
Since $f_*$ maps $\pi_1(T)$ isomorphically onto $A$ any section for 
$p$ induces a section for $A$.

If $f_*(g)=1$ then $gt=tg$, for all $t\in\pi_1(T)$,
since $\pi_1(f|_T)$ is injective.
Hence $\mathrm{Ker}(f_*)$ is contained in the centralizer of $\pi_1(T)$,
and $f_*$ induces an isomorphism of monodromy groups.
\end{proof}

In a similar vein, if an aspherical 4-manifold $Y$ is dominated by 
the total space of a $T$-bundle such that the monodromy 
preserves a flag then 
it is virtually homotopy equivalent to such a $T$-bundle space.


Every solvmanifold of type $\mathbb{E}^4$,
$\mathbb{N}il^3\times\mathbb{E}^1$,
$\mathbb{N}il^4$ or $\mathbb{S}ol^3\times\mathbb{E}^1$
is the total space of a $T$-bundle over $T$, 
and so is 1-dominated by $T$-bundles over hyperbolic bases.
Of these, 
only $T^4$ is dominated by $\mathbb{H}^2\times\mathbb{E}^2$-manifolds.
(However, $\mathbb{N}il^3\times\mathbb{E}^1$-manifolds are dominated by 
$\widetilde{\mathbb{SL}}\times\mathbb{E}^1$-manifolds.)

On the other hand,
no $T$-bundle space can dominate a solvmanifold of type 
$\mathbb{S}ol^4_{m,n}$ with $m\not=n$, 
$\mathbb{S}ol^4_0$ or $\mathbb{S}ol^4_1$.

\section{${\mathbb{H}^3\times\mathbb{E}^1}$ and mapping tori}

The total space $X$ of an $S^1$-bundle over a 3-manifold $M$
has a mixed compact-aspherical or solvable Lie geometry if $M$
is an $\mathbb{S}^3$-, $\mathbb{S}^2\times\mathbb{E}^1$-,
$\mathbb{E}^3$-,  $\mathbb{N}il^3$- or $\mathbb{S}ol^3$-manifold.
If $M$ is a $\mathbb{H}^2\times\mathbb{E}^1$- or
$\widetilde{\mathbb{SL}}$-manifold then $X$ is a $T$-bundle space, 
but is not necessarily geometric.
Domination by such 4-manifolds was considered in \S3 , 
\S4 and \S5 above.

Here we shall consider the remaining possibility: 
dominations by aspherical 4-manifolds which are total spaces of
$S^1$-bundles over $\mathbb{H}^3$-manifolds.
Such manifolds are finitely covered by mapping tori of self-homeomorphisms of
$\mathbb{H}^2\times\mathbb{E}^1$- or $\widetilde{\mathbb{SL}}$-manifolds,
by the Virtual Fibration Theorem of Agol.
This feature is used in Theorem \ref{base and fibre} below.
 
\begin{theorem}
\label{H3xE1}
An aspherical $4$-manifold $Y$ is dominated by a
${\mathbb{H}^3\times\mathbb{E}^1}$-manifold if and only if $\pi_Y$
is virtually a product $G\times\mathbb{Z}$,
where $G$ is a $PD_3$-group.
\end{theorem}

\begin{proof}
Every $\mathbb{H}^3\times\mathbb{E}^1$-manifold is finitely covered 
by a product $B\times{S^1}$,
where $B$ is a closed $\mathbb{H}^3$-manifold.
If $X=B\times{S^1}$ is such a product and $f:X\to{Y}$ 
is an essentially dominating map then the image of $\pi_1(S^1)$ in $\pi_Y$
must be an infinite cyclic direct factor, by Lemma \ref{asph}.
Hence $\pi_Y\cong{G}\times\mathbb{Z}$,  
and the first factor must be a $PD_3$-group,
since $Y$ is aspherical.

Conversely,
if $Y$ is aspherical and $\pi_Y\cong{G}\times\mathbb{Z}$,
then $G$ is a finitely presentable $PD_3$-group,
and so $P=K(G,1)$ is a $PD_3$-complex.
An application of the Atiyah-Hirzebruch spectral sequence
for oriented bordism $\Omega_*(P)$
(or direct desingularization of a geometric 3-cycle representing 
a fundamental class)
shows that $P$ is 1-dominated by a closed orientable  3-manifold.
This is in turn 1-dominated by a closed $\mathbb{H}^3$-manifold \cite{Wa91}.
Hence $Y\simeq{P}\times{S^1}$ is 1-dominated by a
$\mathbb{H}^3\times\mathbb{E}^1$-manifold.
\end{proof}

Hence solvmanifolds of type $\mathbb{E}^4$,
$\mathbb{N}il^3\times\mathbb{E}^1$
and $\mathbb{S}ol^3\times\mathbb{E}^1$ are 1-dominated by 
$\mathbb{H}^3\times\mathbb{E}^1$-manifolds.
Similarly, $\mathbb{H}^2\times\mathbb{E}^2$- and 
$\widetilde{\mathbb{SL}}\times\mathbb{E}^1$-manifolds are so dominated.
However no $\mathbb{N}il^4$-manifold or solvmanifold of type 
$\mathbb{S}ol^4_{m,n}$ with $m\not=n$, 
$\mathbb{S}ol^4_0$ or $\mathbb{S}ol^4_1$ can be so dominated.

The following lemma and theorem extend Theorem \ref{H3xE1} to 
a wider class of manifolds which have two such fibrations.

\begin{lemma}
\label{mtorus}
If an aspherical $4$-manifold $Y$ is dominated by the mapping torus of
a self-homeomorphism of an aspherical $3$-manifold then 
$\pi_Y$ is virtually a semidirect product $K\rtimes\mathbb{Z}$,
with $K$ finitely generated.
\end{lemma}

\begin{proof}
Let $X=M(\varphi)$ be the mapping torus of a self-homeomorphism $\varphi$
of an aspherical 3-manifold $M$,
and let $f:X\to{Y}$ be a map of non-zero degree.
After 
we may assume that $f_*=\pi_1(f)$ is an epimorphism.
Let $K$ be the image of $\pi_1(M)$ in $\pi_Y$.
Then $K$ is finitely generated and $\pi_Y/K$ is cyclic.

Suppose that $\pi/K$ is finite.
Then after passing to a finite cover, if necessary,
we may assume that $K=\pi_Y$.
But $f_*$ induces a homomorphism of spectral sequences 
from the LHS spectral sequence for $\pi_X$ as an extension 
of $\mathbb{Z}$ by $\pi_1(M)$ and for $\pi_Y$ as an extension of 1 by $\pi_Y$.
Since $H_4(X)\cong{H_1(\mathbb{Z};H_3(\pi_1(M)))}$,
while $H_1(1;H_3(\pi_Y))=0$,
it follows that $f$ has degree 0,
contrary to hypothesis.
Therefore $\pi_Y/K\cong\mathbb{Z}$,
and so $\pi_Y$ is a semidirect product ${K}\rtimes\mathbb{Z}$.
\end{proof}

\begin{theorem}
\label{base and fibre}
An aspherical $4$-manifold $Y$ is essentially dominated by the mapping torus of
a self-homeomorphism of an aspherical $3$-manifold $M$ 
such that $\pi_1(M)$ has non-trivial centre if and only if $Y$ 
is homotopy equivalent to such a mapping torus.
\end{theorem}

\begin{proof}
Let $M$ be a 3-manifold such that $\pi_1(M)$ has non-trivial centre $C$,
and let $X=M(\varphi)$ be the mapping torus of a self-homeomorphism 
$\varphi$ of $M$.
If $f:X\to{Y}$ is a map with non-zero degree 
and such that $f_*=\pi_1(f)$ is an epimorphism
then $A=f_*C$ and $K=f_*\pi_1(M)$  are normal subgroups of $\pi_Y$
such that $A\leq{K}$.
Moreover, $K$ is finitely generated and $\pi_Y/K\cong\mathbb{Z}$, 
by Lemma \ref{mtorus}.

If $C\cong\mathbb{Z}$ then $A\cong\mathbb{Z}$ also, 
for otherwise $f$ would factor through $K(\pi_X/C,1)$, 
and so have degree 0, by Lemma \ref{asph}.
Hence $\beta_i^{(2)}(\pi_Y)=0$ for all $i$, 
and so $\chi(Y)=0$,
by a result of Gromov.
(See \cite[Corollary 2.3.1]{Hi}.)
Therefore $K$ is a $PD_3$-group \cite[Theorem 4.5]{Hi}.
Since $K$ has non-trivial centre it is the fundamental group
of an aspherical Seifert fibred 3-manifold, and the claim follows.

If $C$ has rank $>1$ then $M=T^3$, so $X$ is a solvmanifold and the
claim follows from the discussion in the first three paragraphs of \S4.

The condition is clearly sufficient, since we may take $f=id_Y$.
\end{proof}

\section{$\mathbb{H}^2\times\mathbb{H}^2$ and surface bundles}

A {\it Surface bundle \/} is a bundle projection $p:X\to{B}$ with fibre $F$, 
where $B$ and $F$ are closed surfaces,  $\chi(B)\leq0$ and $\chi(F)<0$.
A {\it Surface bundle space\/} is the total space of a Surface bundle,
and a {\it Surface bundle group\/} is the fundamental group 
of a Surface bundle space.
Such Surface bundles may be partitioned into three types.
Type I consists of such bundles for which the monodromy has infinite image, 
but is not injective,
type II are those which are virtually products and 
type III have injective monodromy \cite{Jo94}.
The fibration is essentially unique when the bundle is of type I; 
product bundles have only the two obvious bundle projections, 
and a surface bundle space $X$ with a fibration of type III 
may have many inequivalent fibrations.
(See the ``Johnson trichotomy" in  \cite[Chapter 5.2]{Hi}.)

If a Surface bundle space $X$ is geometric but is not a $\mathbb{H}^4$-manifold
then either $\chi(X)=0$ (so $B=T$) and the geometry is $\mathbb{H}^2\times\mathbb{E}^2$ or $\mathbb{H}^3\times\mathbb{E}^1$,
or $X$ is a reducible $\mathbb{H}^2\times\mathbb{H}^2$-manifold
\cite[Theorems 9.10, 13.5 and 13.6]{Hi}.
Every reducible $\mathbb{H}^2\times\mathbb{H}^2$-manifold 
is finitely covered by a product $B\times{F}$, 
while every $\mathbb{H}^3\times\mathbb{E}^1$-manifold is finitely covered
by a Surface bundle space of type I, with base $T$.
No Surface bundle space has the geometry $\mathbb{H}^2(\mathbb{C})$
\cite[Corollary 13.7.2]{Hi}, 
and it is believed that $\mathbb{H}^4$ is also impossible.

Pullback along a degree-1 map of bases induces a degree-1 map 
of Surface bundle spaces.
If the original bundle is of type I or III and the base change map
is not a homotopy equivalence then the induced bundle is of type I.
If the monodromy fixes a separating curve in the fibre $F$ then 
we may construct a degree-1 map by crushing one of the two 
complementary regions of this curve in $F$ to a point.
In this way we may show that every Surface bundle space of type I or II
is 1-dominated by Surface bundle spaces of types I and III.
On the other hand, we shall see that Surface bundle spaces of type II 
can only dominate bundle spaces of the same type.

\begin{lemma}
\label{degree 0}
Let $f:\pi\to{G}$ be an epimorphism of $PD_4$-groups,
where $\pi$ has a normal subgroup $\kappa$ such that $\kappa$ 
and $\pi/\kappa$ are $PD_2$-groups.
If $f(\kappa)$ or $G/f(\kappa)$ is virtually free
then $f$ has degree $0$.
\end{lemma}

\begin{proof}
If $f(\phi)=1$ then $f$ factors though the $PD_2$-group $\rho=\pi/\kappa$, 
and so $f$ has degree 0, by Lemma \ref{asph}.
In general,  $f$ induces a map between the LHS spectral sequences for 
$\mathbb{Q}$-homology of $\pi$ and $G$ as extensions or $\rho$ by $\kappa$ 
and of $G/f(\kappa)$ by $f(\kappa)$, respectively. 
If $f(\kappa)$ or $G/f(\kappa)$ is virtually free then the homomorphisms 
between $E^2_{p,q}$ terms  with $p+q=4$ either have domain 0 or codomain 0.
Hence they are all trivial, and so $H_4(f)=0$.
\end{proof}

We have a nearly complete understanding of the possibilities when $X$ is of type II
and $Y$ is aspherical.

\begin{theorem}
\label{BxF}
Let $X=B\times{F}$, where $B$ and $F$ are hyperbolic surfaces,
and let $f:X\to{Y}$ be an essentially dominating map with aspherical codomain $Y$.
If $\chi(Y)\not=0$ or if $Y$ is the total space of a bundle with base or fibre $T$
then $Y$ is homotopy equivalent to a product of surfaces. 
\end{theorem}

\begin{proof}
Let $G$ and $N$ be the images of $\pi_1(B)$ and $\pi_1(F)$, respectively.
Then $G$ and $N$ are finitely generated normal subgroups of $\pi_Y$
such that $gn=ng$ for all $g\in{G}$ and $n\in{N}$, and  $\pi_Y=GN$.
Hence $gh=hg$ for all $g\in{G\cap{N}}$ and all $h\in\pi_Y$,
and so $G\cap{N}$ is a subgroup of the centre of $\pi_Y$.
Let $\overline\pi=\pi_Y/(G\cap{N})$.
Then $\overline\pi\cong\pi/G\times\pi/N$.

If $G\cap{N}\not=1$ then $\beta_i^{(2)}(Y)=0$ for all $i$, and so $\chi(Y)=0$.
Since we are assuming in this case that $Y$ is a bundle space,
$G\cap{N}\cong\mathbb{Z}^r$, for some $r>0$.
Then $\overline\pi$ is $FP$, and hence so are $\pi/G$ and $\pi/N$.
Consideration of the LHS spectral sequence for $\pi_Y$ as an
extension of $\overline\pi$ by $\mathbb{Z}^r$ then shows that
$H^{4-r}(\overline\pi;\mathbb{Z}[\overline\pi])\cong\mathbb{Z}$.
(Compare \cite[Theorem 3.10]{Hi}.)
It then follows from the K\"unneth Theorem for
$\overline\pi\cong\pi/G\times\pi/N$ that either $\pi/G$ or $\pi/N$ 
 two ends or one is finite.
This contradicts Lemma \ref{degree 0}.

Therefore $G\cap{N}=1$, so $\pi_Y\cong{G}\times{N}$, 
and so $G$ and $N$ are finitely presentable.
Hence they are in fact $FP_3$ (each being a quotient of an $FP$
group by a finitely presentable group).
Neither $G$ nor $N$ is trivial or $\mathbb{Z}$, by Lemma \ref{degree 0},
and so each is a $PD_2$-group \cite[theorem 3.10]{Hi}.
\end{proof}

We can push the argument for this theorem a little further if 
$\chi(Y)=0$ and $G\cap{N}\not=1$.
We note first that $c.d.G\leq3$ and $c.d.N\leq3$, 
since $G$ and $N$ each have infinite index in $\pi_Y$.
Hence $G\cap{N}$ has rank $r\leq3$.
The argument of the theorem shows that $G\cap{N}$ is not finitely generated
Hence  it has rank $r<c.d.G\cap{N}<4$
and so $r=1$ or 2.

Suppose that $r=2$.
Then  $c.d.G\cap{N}=3$, 
since $G\cap{N}$ is not finitely generated.
But $c.d.G=3$ also, and so $G$ is abelian \cite[Theorem 8.8 ]{Bi}.
Since $G$ is finitely generated,
$G\cap{N}$ must be finitely generated also.
This is a contradiction, and so $r=1$.
It remains an open question whether the centre of a $PD_n$-group
must be finitely generated, if $n>3$.

General results on essential dominations of aspherical 4-manifolds 
by Surface bundle spaces of types I or III appear to be hard to find.
When $Y$ is also a bundle space we should ask whether
such a domination need be compatible with a fibration,
so that the normal subgroup $N=f_*\pi_1(F)$ of $\pi_Y$ 
and the quotient $\pi_Y/N$ are each $PD_2$-groups.

\begin{lemma}
\label{I-II}
Let $E$ be a Surface bundle group.
If $N$ is a finitely generated infinite normal subgroup of infinite index
 in $E$ then $N$ is either a $PD_2$-group or is an extension of a $PD_2$-group
by a free group of infinite rank.
Moreover either $N$ or $E/N$ has one end.
\end{lemma}

\begin{proof}
Let $p:E\to{G}$ be the projection of $E$ onto $G=\pi_1(B)$, 
with kernel $K=\pi_1(F)$.
If $p(N)\not=1$ then $[G:p(N)]<\infty$, and $\mathrm{Ker}(p|_N)$ 
has infinite index in $K$. 
Hence either $p|_N$ is a monomorphism or $\mathrm{ker}(p|_N)$ 
is free of infinite rank. (This uses the fact that $F$ is hyperbolic!)

The second assertion follows from consideration of the LHS spectral sequence for
$E$ as an extension of $E/N$ by $N$.
Since $E/N$ and $N$ are finitely generated and infinite,
$H^p(E/N;H^q(N;\mathbb{Z}[E]))=0$ if $q=0$ or if $p=0$ and $q=1$,
while 
\[H^1(E/N;H^1(N;\mathbb{Z}[E]))\cong
{H^1(E/N;\mathbb{Z}[E/N])}\otimes{H^1(N;\mathbb{Z}[N])}.
\]
Since $H^i(E;\mathbb{Z}[E])=0$ for $i\leq3$,
a spectral sequence corner argument shows that 
$H^1(E/N;H^1(N;\mathbb{Z}[E]))=0$.
Since the terms in the tensor product are free abelian groups,
at least one must be 0.
\end{proof}

If $\beta_1(E)>\beta_1(B)$ then there are finitely generated
normal subgroups $N$ such that $E/N\cong\mathbb{Z}$
and $\mathrm{ker}(p|_N)$ is nontrivial,
and if, moreover, $\chi(E)\not=0$ then such a subgroup $N$
cannot be $FP_2$ \cite{FV19}.

\section{domination of aspherical geometric manifolds by 
surface bundle spaces}

Suppose that $X$ is a Surface bundle space which dominates a  
$\mathbb{H}^2\times\mathbb{E}^2$-manifold $Y$.
After passing to finite covers we may assume that $Y$ is a product $C\times{T}$,
where $C$ is a hyperbolic surface.
If $N$ is a non-trivial finitely generated normal subgroup of infinite index in
$\pi_Y$ then either $N$ has finite index in $\pi_1(C)$ or $\pi_1(T)$, 
or $N\cong\mathbb{Z}<\pi_1(T)$ or $N\cap\pi_1(C)$ has finite index 
in $\pi_1(C)$ and $\pi_Y/N$ is virtually $\mathbb{Z}$.
The latter two possibilities are ruled out by Lemma \ref{degree 0},
and so $f$ must be compatible with one of the two projections from $Y$
to $C$ or $T$.

If $Y$ is a $\widetilde{\mathbb{SL}}\times\mathbb{E}^1$-manifold
then after passing to finite covers we may assume that $Y$ is the total space 
of a $T$-bundle over a hyperbolic surface.
A similar argument applies, 
except that in this case the bundle fibration is unique and the
only normal $PD_2$-groups in $\pi_Y$ are subgroups of finite index in $\pi_1(T)$.
Hence $f$ must be compatible with the fibration of $Y$.
In particular,
no $\widetilde{\mathbb{SL}}\times\mathbb{E}^1$-manifold
is dominated by a product of surfaces.

Suppose now that $Y$ is a solvmanifold.
The image of $\pi_1(F)$ in $\pi_Y$ is a finitely generated normal subgroup,
and is torsion-free polycyclic.
It must have Hirsch length 2, by Lemma \ref{degree 0},
and so $\pi_1(f)$ factors though the group of a $T$-bundle over $B$.
Every solvmanifold with geometry $\mathbb{E}^4$,
$\mathbb{N}il^3\times\mathbb{E}^1$,
$\mathbb{N}il^4$ or $\mathbb{S}ol^3\times\mathbb{E}^1$
is the total space of a $T$-bundle over $T$.
Hence such manifolds are 1-dominated by Surface bundle spaces.

There are  epimorphisms from Surface bundle groups 
to semidirect products $\mathbb{Z}^3\rtimes\mathbb{Z}$,
with $\pi_1(F)$ mapping onto $\mathbb{Z}^3$.
However all such maps have degree 0,
and no solvmanifold of type $\mathbb{S}ol^4_{m,n}$ 
or $\mathbb{S}ol^4_0$ is dominated by a Surface bundle space.
Similarly, no $\mathbb{S}ol^4_1$-manifold is dominated by a Surface bundle space.

If $Y$ is aspherical then the image of $\pi_1(F)$ in $\pi_Y$ is
a finitely generated normal subgroup, and $G=\pi_Y/K$ is finitely presentable.
Lemma \ref{degree 0} implies that $K$ is not a free group
and that $G$ is not virtually free, so $[\pi_Y:K]=\infty$.
Hence $c.d.K=2$ or 3.

\end{document}